\documentclass[10pt,oneside]{amsart}

\usepackage[
paper=a4paper,
headsep=15pt,text={136mm,221mm},centering
]{geometry}
\usepackage{kbordermatrix}
\usepackage{multirow} 
\usepackage{tabstackengine}

\def\VR{\kern-\arraycolsep\strut\vrule &\kern-\arraycolsep}
\def\vr{\kern-\arraycolsep & \kern-\arraycolsep}

\usepackage{graphicx, pinlabel, color, xcolor}
\usepackage{blkarray}
\usepackage{array, boldline, makecell, booktabs}
\usepackage[bookmarks]{hyperref}
\usepackage{amssymb,amsxtra}
\usepackage{graphicx,xcolor}
\usepackage{float,array,calc,booktabs,stackrel}
\usepackage{enumitem}\setlist[enumerate,1]{font=\upshape}
\usepackage{tikz}
\usetikzlibrary{
  cd,
  calc,
  positioning,
  arrows,
  decorations.pathreplacing,
  decorations.markings,
}
\usepackage[mode=buildnew]{standalone}

\definecolor{todo-background-color}{gray}{0.95}
\usepackage[
  textwidth=1.1in,
  backgroundcolor=todo-background-color,
  bordercolor=black,
  linecolor=black,
  textsize=footnotesize
]{todonotes}

\iftrue
    
    \makeatletter
    \def\@settitle{%
      \vspace*{-10pt}
      \begin{flushleft}%
        \LARGE\bfseries
        \strut\@title\strut
      \end{flushleft}%
    }
    \def\@setauthors{%
      \begingroup
      \def\thanks{\protect\thanks@warning}%
      \trivlist
      \raggedright
      \large \@topsep27\p@\relax
      \advance\@topsep by -\baselineskip
    \item\relax
      \author@andify\authors
      \def\\{\protect\linebreak}%
      \authors
      \ifx\@empty\contribs
      \else
      ,\penalty-3 \space \@setcontribs
      \@closetoccontribs
      \fi
      \normalfont
      \endtrivlist
      \endgroup
    }
    \def\@setaddresses{\par
      \nobreak \begingroup
      \small\raggedright
      \def\author##1{\nobreak\addvspace\smallskipamount}%
      \def\\{\unskip, \ignorespaces}%
      \interlinepenalty\@M
      \def\address##1##2{\begingroup
        \par\addvspace\bigskipamount\noindent
        \@ifnotempty{##1}{(\ignorespaces##1\unskip) }%
        {\ignorespaces##2}\par\endgroup}%
      \def\curraddr##1##2{\begingroup
        \@ifnotempty{##2}{\nobreak\noindent\curraddrname
          \@ifnotempty{##1}{, \ignorespaces##1\unskip}\/:\space
          ##2\par}\endgroup}%
      \def\email##1##2{\begingroup
        \@ifnotempty{##2}{\nobreak\noindent E-mail address%
          \@ifnotempty{##1}{, \ignorespaces##1\unskip}\/:\space
          \ttfamily##2\par}\endgroup}%
      \def\urladdr##1##2{\begingroup
        \def~{\char`\~}%
        \@ifnotempty{##2}{\nobreak\noindent\urladdrname
          \@ifnotempty{##1}{, \ignorespaces##1\unskip}\/:\space
          \ttfamily##2\par}\endgroup}%
      \addresses
      \endgroup
      \global\let\addresses=\@empty
    }
    \def\@setabstracta{%
      \ifvoid\abstractbox
      \else
      \skip@17pt \advance\skip@-\lastskip
      \advance\skip@-\baselineskip \vskip\skip@
      \box\abstractbox
      \prevdepth\z@ 
      \vskip-28pt
      \fi
    }
    \renewenvironment{abstract}{%
      \ifx\maketitle\relax
      \ClassWarning{\@classname}{Abstract should precede
        \protect\maketitle\space in AMS document classes; reported}%
      \fi
      \global\setbox\abstractbox=\vtop \bgroup
      \normalfont\small
      \list{}{\labelwidth\z@
        \leftmargin0pc \rightmargin\leftmargin
        \listparindent\normalparindent \itemindent\z@
        \parsep\z@ \@plus\p@
        
      }%
    \item[\hskip\labelsep\bfseries\abstractname.]%
    }{%
      \endlist\egroup
      \ifx\@setabstract\relax \@setabstracta \fi
    }

    \def\ps@headings{\ps@empty
      \def\@evenhead{%
        \setTrue{runhead}%
        \normalfont\scriptsize
        \rlap{\thepage}\hfill
        \def\thanks{\protect\thanks@warning}%
        \leftmark{}{}}%
      \def\@oddhead{%
        \setTrue{runhead}%
        \normalfont\scriptsize
        \def\thanks{\protect\thanks@warning}%
        \rightmark{}{}\hfill \llap{\thepage}}%
      \let\@mkboth\markboth
    }\ps@headings

    \def\section{\@startsection{section}{1}%
      \z@{-1.4\linespacing\@plus-.5\linespacing}{.8\linespacing}%
      {\normalfont\bfseries\Large}}
    \def\subsection{\@startsection{subsection}{2}%
      \z@{-.8\linespacing\@plus-.3\linespacing}{.5\linespacing\@plus.2\linespacing}%
      {\normalfont\bfseries\large}}
    \def\subsubsection{\@startsection{subsubsection}{3}%
      \z@{.7\linespacing\@plus.2\linespacing}{-1.5ex}%
      {\normalfont\itshape}}
    \def\paragraph{\@startsection{paragraph}{4}%
      \z@{.7\linespacing\@plus.2\linespacing}{-1.5ex}%
      {\normalfont\itshape}}
    \def\@secnumfont{\bfseries}

    \renewcommand\contentsnamefont{\bfseries}
    \def\@starttoc#1#2{\begingroup
      \setTrue{#1}%
      \par\removelastskip\vskip\z@skip
      \@startsection{}\@M\z@{\linespacing\@plus\linespacing}%
      {.5\linespacing}{
        \contentsnamefont}{#2}%
      \ifx\contentsname#2%
      \else \addcontentsline{toc}{section}{#2}\fi
      \makeatletter
      \@input{\jobname.#1}%
      \if@filesw
      \@xp\newwrite\csname tf@#1\endcsname
      \immediate\@xp\openout\csname tf@#1\endcsname \jobname.#1\relax
      \fi
      \global\@nobreakfalse \endgroup
      \addvspace{32\p@\@plus14\p@}%
      \let\tableofcontents\relax
    }
    \def\contentsname{Contents}
    \def\l@section{\@tocline{2}{.5ex}{0mm}{5pc}{}}
    \def\l@subsection{\@tocline{2}{0pt}{2em}{5pc}{}}
    \makeatother

\fi

\def\to{\mathchoice{\longrightarrow}{\rightarrow}{\rightarrow}{\rightarrow}}
\makeatletter
\newcommand{\shortxra}[2][]{\ext@arrow 0359\rightarrowfill@{#1}{#2}}
\def\longrightarrowfill@{\arrowfill@\relbar\relbar\longrightarrow}
\newcommand{\longxra}[2][]{\ext@arrow 0359\longrightarrowfill@{#1}{#2}}

\makeatother

\makeatletter
\def\addtagsub#1{\let\oldtf=\tagform@\def\tagform@##1{\oldtf{##1}\hbox{$_{#1}$}}}
\makeatother

\makeatletter
\def\Nopagebreak{\@nobreaktrue\nopagebreak}
\makeatother

\makeatletter
\newtheoremstyle{theorem-giventitle}
        {}{}              
        {\itshape}                      
        {}                              
        {\bfseries}                     
        {.}                             
        {\thm@headsep}                             
        {\thmnote{\bfseries#3}}
\newtheoremstyle{theorem-givenlabel}
        {}{}              
        {\itshape}                      
        {}                              
        {\bfseries}                     
        {.}                             
        {\thm@headsep}                             
        {\thmname{#1}~\thmnumber{#3}\setcurrentlabel{#3}}
\newtheoremstyle{definition-giventitle}
        {}{}              
        {}                      
        {}                              
        {\bfseries}                     
        {.}                             
        {\thm@headsep}                             
        {\thmnote{\bfseries#3}}
\def\setcurrentlabel#1{\gdef\@currentlabel{#1}}
\makeatother

\newtheorem{theorem}{Theorem}[section]

\newtheorem{proposition}[theorem]{Proposition}

\newtheorem{thmx}{Theorem}

\theoremstyle{definition}
\newtheorem{definition}[theorem]{Definition}

\newtheorem{remark}[theorem]{Remark}

\newtheorem*{case2'}{Case 2$'$}

\theoremstyle{theorem-giventitle}
\newtheorem{theorem-named}{}
\theoremstyle{theorem-givenlabel}
\newtheorem{theorem-labeled}{Theorem}

\theoremstyle{definition-giventitle}
\newtheorem{definition-named}{}
\newtheorem{conjecture-named}{}
\newtheorem{case-named}{}

\numberwithin{equation}{section}

\def\d{\partial}
\def\Z{\mathbb{Z}}

\def\Q{\mathbb{Q}}

\newcommand{\Id}{\operatorname{Id}}

\newcommand{\rk}{\operatorname{rank}}

\makeatletter

\begin{document}

\title{Pretzel links, mutation, and the slice-ribbon conjecture}

\author{Paolo Aceto}
\address{Mathematical Institute University of Oxford, Oxford, United Kingdom}
\email{paoloaceto@gmail.com}
\urladdr{http://www.maths.ox.ac.uk/people/paolo.aceto/}
\author{Min Hoon Kim}
\address{Department of Mathematics, Chonnam National University, Gwangju  61186, Republic of Korea}
\email{kminhoon@gmail.com}
\urladdr{http://minhoon.kim}
\author{JungHwan Park}
\address{School of Mathematics, Georgia Institute of Technology, Atlanta, GA, USA}
\email{junghwan.park@math.gatech.edu }
\urladdr{http://people.math.gatech.edu/~jpark929/}
\author{Arunima Ray}
\address{Max Planck Institut f\"{u}r Mathematik, Bonn, Germany}
\email{aruray@mpim-bonn.mpg.de }
\urladdr{http://people.mpim-bonn.mpg.de/aruray/}
\def\subjclassname{\textup{2020} Mathematics Subject Classification}
\expandafter\let\csname subjclassname@1991\endcsname=\subjclassname
\expandafter\let\csname subjclassname@2000\endcsname=\subjclassname
\subjclass{%
  57K40, 
  	57K10} 


\newcommand{\blue}{\textcolor{blue}}

\newcommand{\JP}[1]{[JP] \blue{#1}}

\begin{abstract}
Let $p$ and $q$ be distinct integers greater than one. 
We show that the $2$-component pretzel link $P(p,q,-p,-q)$ is not slice, even though it has a ribbon mutant, by using $3$-fold branched covers and an obstruction based on Donaldson's diagonalization theorem. As a consequence, we prove the slice-ribbon conjecture for $4$-stranded $2$-component pretzel links.  
\end{abstract}

\maketitle

\section{Introduction}
A link $L\subset S^3$ is \emph{slice} if it bounds a collection of disjoint, smoothly embedded disks in~$D^4$, and it is \emph{ribbon} if it bounds a collection of immersed disks in $S^3$ with only ribbon singularities. 
Ribbon links are slice links since we can remove ribbon singularities by pushing them into the interior of $D^4$. 
The \emph{slice-ribbon conjecture}, an outstanding open problem due to Fox \cite{Fox:1961-1}, states that all slice knots are ribbon.  In this article, we address the natural generalization to 2-component links.

A significant positive result, due to Lisca \cite{Lisca:2007-1}, shows that the slice-ribbon conjecture holds for $2$-bridge knots. Specifically, Lisca showed that a $2$-bridge knot is ribbon if and only if its double branched cover, which is a rational homology sphere, bounds a rational homology ball. 
The lattice-theoretic obstruction coming from Donaldson's theorem was combined with Heegaard-Floer invariants by Greene
and Jabuka, who showed that the slice-ribbon conjecture is true for all $3$-stranded pretzel knots $P(p,q,r)$ with $p,q,r$ odd~ \cite{Greene-Jabuka:2011-1}. 
In \cite{Lecuona:2012-1}, Lecuona used similar techniques as Lisca to prove the slice-ribbon conjecture for a large infinite family of $3$-stranded Montesinos knots.  
Related questions were addressed by~\cite{Lecuona:2015-1,Miller:2017-1,Choe-Park:2018-1}. 

We prove the slice-ribbon conjecture for $4$-stranded $2$-component pretzel links. Recall that a \emph{$4$-stranded pretzel link} $P(p,q,r,s)$ is a link admitting a diagram as in Figure \ref{fig:pretzel_link} where each $p$, $q$, $r$, $s$ denotes the number of half twists in the box.

\begin{thmx}\label{thm:main}
The slice-ribbon conjecture holds for $4$-stranded $2$-component pretzel links.
\end{thmx}

The proof has two parts. The first is the following theorem, derived from the work of \cite{Aceto:2020-1} where he follows similar methods as Lisca. In the upcoming paper of Aceto \cite{Aceto:2019-1}, he applies the obstructions from \cite{Aceto:2020-1} to a larger family of 2-component links.

\begin{theorem}\label{thm:ribbon-up-to-mutation}
Suppose that the $4$-stranded $2$-component pretzel link $L$ is slice. 
Then $L$ is isotopic to either $P(p,q,-p,-q)$ or $P(p,-p,q,-q)$ for some integers $p$ and~$q$.
\end{theorem}




Observe that the link $P(p,-p, q,-q)$ is ribbon, as shown in Figure~\ref{fig:ribbon_move}. Additionally, $P(p,q,-p,-q)$ is ribbon whenever $|p| = |q|$ or either $|p|$ or $|q|$ is one. It remains to consider links of the form $P(p,q,-p,-q)$. 
In the proof of Theorem~\ref{thm:ribbon-up-to-mutation}, we show that the $2$-component link $P(0,q,0,-q)$ is not slice for any $q$. 
Additionally, the links $P(p,q,r,s)$ and $P(q,r,s,p)$ are isotopic. Consequently, the following result completes the proof of Theorem~\ref{thm:main}.

\begin{theorem}\label{thm:mutants-not-concordant}
Let $p$ and $q$ be integers such that $1<p<q$. 
Then the $2$-component pretzel link $P(p,q,-p,-q)$ is not slice.
\end{theorem}

\begin{figure}[htb]
\includegraphics[scale=1]{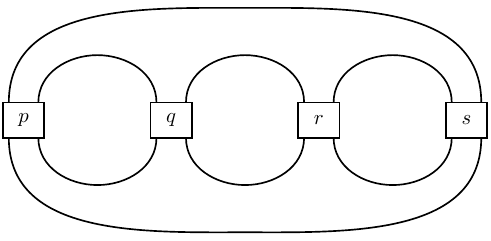}
\caption{The pretzel link $P(p,q,r,s)$.}\label{fig:pretzel_link}
\end{figure}

An embedded $2$-sphere $S$ in $S^3$ is called a \emph{Conway sphere} for a link $L$ if $S$ meets $L$ transversely in exactly four points. 
By cutting $S^3$ along $S$ and regluing via an orientation-preserving involution on $S$ which preserves $L\,\cap\,S$ setwise and does not fix any points in $L\,\cap\,S$, we obtain a new link in $S^3$ with the same number of components as $L$. 
This new link is called a \emph{mutant} of $L$. 
If $L$ is an oriented link and the orientation of $L$ is preserved by the involution on $S$, the new link is called a \emph{positive mutant} of $L$.

\begin{figure}[htb]
\includegraphics[scale=1]{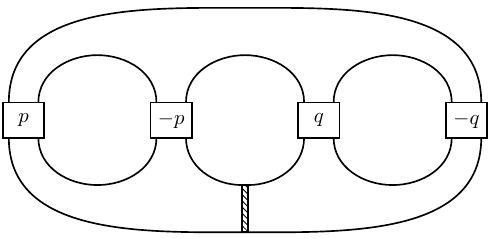}
\caption{The pretzel link $P(p,-p,q,-q)$ is ribbon via the band move shown above.}\label{fig:ribbon_move}
\end{figure}

Note that the links $P(p,q,-p,-q)$ and $P(p,-p,q,-q)$ are mutants, and indeed are positive mutants under a choice of orientation. In other words, Theorem~\ref{thm:mutants-not-concordant} shows that the link $P(p,q,-p,-q)$ is not slice, even though it has a slice mutant, namely $P(p,-p,q,-q)$. This is often difficult, since mutant links have diffeomorphic double branched covers~\cite{Viro:1976-1}. Kearton, based on the work of Livingston~\cite{Livingston:1983-1}, gave the first example of non-concordant mutant knots~\cite{Kearton:1989-1}. Since then, work has focused on distinguishing between positive mutants. Kirk and Livingston constructed a pair of non-concordant positive mutants in~\cite{Kirk-Livingston:1999-2}, and later gave an infinite family of examples in~\cite{Kirk-Livingston:2001-1}. These, as well as later work by~\cite{Herald-Kirk-Livingston:2010-1, Kim-Livingston:2005-1, Miller:2017-2}, use Casson-Gordon invariants and twisted Alexander polynomials associated with higher-order branched covers. 
Examples of non-concordant positive mutant links were constructed by Cha using Milnor's invariants~\cite{Cha:2006-1}. 
The computations involved in these three approaches - Milnor's invariants, Casson-Gordon invariants, and twisted Alexander polynomials -  are rather daunting in general, which is reflected in the paucity of known examples of non-concordant mutant knots and links. 

In the context of proving the slice-ribbon conjecture for families of links, mutant links did not arise in the work of Lisca since any lens space occurs as a branched double cover of a unique link in $S^3$~\cite{Hodgson-Rubinstein:1985-1}. The papers of Greene-Jabuka, Lecuona, and Choe-Park mentioned above sidestep mutant knots since there is no non-trivial mutation within the families they consider. However, the problem remains for more general families of knots and links if using double branched covers, as seen in the work of~\cite{Long:2014-1,Bryant:2017-1}.

In this article, we use a novel technique to distinguish mutant links in concordance. Specifically, we combine the use of higher-order branched covers with obstructions based on Donaldson's theorem. Even more specifically, we follow Lisca's strategy from~\cite{Lisca:2007-1} as follows. Let $L$ be an $n$-component slice link and let $\Sigma_{m^r}(L)$ denote the $m^r$-fold cyclic branched cover of $S^3$ along $L$. First, we generalize a result of Casson-Gordon by showing that $\Sigma_{m^r}(L)$ is a rational homology $\#_{i=1}^{(m^r-1)(n-1)}S^1\times S^2$ and bounds a rational homology $\natural_{i=1}^{(m^r-1)(n-1)}S^1\times D^3$, for any prime $m$ and integer $r$ (Proposition~\ref{prop:betti-number-two-piece}). Combined with Donaldson's theorem, this provides a slicing obstruction (Theorem~\ref{thm:obstruction}), which states that if $L$ is slice and $\Sigma_{m^r}(L)$ bounds a smooth, simply connected, positive semidefinite $4$-manifold $X$ such that $\operatorname{rank}(Q_X)=b_2(X)-b_1(\Sigma_{m^r}(L))$, then there is a morphism of integral lattices $$(H_2(X),Q_X)\rightarrow (\Z^{\operatorname{rank}(Q_X)}, \Id).$$ In the previous work mentioned above, using double branched covers of $2$-bridge knots and Montesinos links, there is a canonical choice of such a manifold $X$, obtained as a plumbing. However, in our case, since we use higher-order branched covers, we are required to construct these $4$-manifolds \textit{ad hoc}. In Section~\ref{sec:proof}, we explicitly construct the $3$-fold cover of $S^3$ branched along a pretzel link of the form $P(p,q,-p,-q)$ with $1<p<q$ and a $4$-manifold $X$ satisfying the hypotheses of Theorem~\ref{thm:obstruction}. Then we show that there is no such lattice morphism $(H_2(X),Q_X)\rightarrow (\Z^{\operatorname{rank}(Q_X)}, \Id).$  This completes the proof.

\begin{remark}Note that there is a band move from the 2-component link $P(p,q,-p,-q)$ to the knot $P(p\pm 1,q,-p,-q)$. This observation, along with Theorem~1.4 of \cite{Miller:2017-2}, gives a different proof of Theorem~\ref{thm:mutants-not-concordant}, for a particular subfamily of 4-strand 2-component pretzel links. 
\end{remark}

\begin{remark}It is natural to ask if our method can be applied to other families of knots and links. When looking at pretzel links with more than four strands, several complications arise, which make it difficult to implement our technique in practice. For instance, the number of possible mutant links corresponding to a fixed unordered set of coefficients increases rapidly as the number of strands increases. Also, the description of the 3-fold cover is more involved and in some cases it is not clear to us how to construct semidefinite fillings.\end{remark}

%
%
%

\subsection*{Acknowledgements} This project started when PA was visiting the Korea Institute for Advanced Study. PA wishes to thank the Korea Institute for Advanced Study for their hospitality. Part of this work was done when MHK was visiting the Max-Planck-Institut f\"{u}r Mathematik. MHK was
partly supported by the POSCO TJ Park Science Fellowship and NRF grant 2019R1A3B2067839. Frank Swenton's Kirby calculator was extremely useful in the early stages of this project. We are also grateful to the anonymous referee for their thoughtful suggestions.

\section{A link slicing obstruction and the proof of Theorem~\ref{thm:ribbon-up-to-mutation}}\label{sec:obstruction}

The main goals of this section are to prove Theorem~\ref{thm:obstruction} which gives a  link slicing obstruction and to prove Theorem~\ref{thm:ribbon-up-to-mutation}.

A \emph{homology bouquet of $n$ circles} is a finite CW-complex $X$ with a homology equivalence $f\colon \bigvee_{i=1}^n S^1\to X$. We similarly define a \emph{rational homology bouquet of $n$ circles} by requiring $f$ to be a rational homology equivalence. With a meridian map, the slice disk complement of an $n$-component slice link is a homology bouquet of $n$ circles. (Meridian maps are homology equivalences by Alexander duality.) For each integer $k>0$ and a homology bouquet of $n$ circles $X$, let $X_k$ be the $k$-fold cyclic cover of $X$ corresponding to 
\begin{center}$\pi_1(X)\to H_1(X)\cong H_1(\bigvee_{i=1}^n S^1)\to \Z_k$\end{center}
 where the first map is the Hurewicz map, the second map is induced from $f$, and the last map sends each circle summand to $1\in \Z_k$. 
\begin{proposition}\label{prop:Levine}If $X$ is a homology bouquet of $n$ circles, then $X_{m^r}$ is a rational homology bouquet of $m^r(n-1)+1$ circles for every prime $m$. 
\end{proposition}
\begin{proof} Let $Y=\bigvee_{i=1}^n S^1$ and $f\colon Y\to X$ be a homology equivalence. By changing $X$ to the mapping cylinder of $f$, we may assume that $f\colon Y\to X$ is an injective cellular map and choose a finite, relative CW-complex structure of $(X,Y)$. Let $Y_{m^r}$ be the pull-back covering of $X_{m^r}$ on $Y$. By definition, $Y_{m^r}$ is the covering of $Y$ corresponding to $\pi_1(Y)\to \Z_{m^r}$ which sends each circle summand of $Y$ to $1$, and hence $Y_{m^r}$ is homeomorphic to $\bigvee_{i=1}^{m^r(n-1)+1}S^1$.

Consider a cellular chain complex $C_*=C_*(X_{m^r},Y_{m^r})$ of finitely generated free $\Z[\Z_{m^r}]$-modules. Note that $C_*\otimes_{\Z[\Z_{m^r}]}\Z_{m}$ is the cellular chain complex $C_*(X,Y;\Z_m)$. Hence, by Alexander duality,
\[H_i(C_*\otimes_{\Z[\Z_{m^r}]}\Z_m)\cong H_i(X,Y;\Z_m)\cong 0\]
for all~$i$. By Levine's chain homotopy lifting argument (see \cite[p.\ 89 and p.\ 95]{Levine:1994-1} and \cite[Lemma~3.2]{Cha:2010-1}),
\[H_i(C_*\otimes_{\Z}\Z_m)\cong H_i(X_{m^r},Y_{m^r};\Z_m) \cong 0\]
for every~$i$. By the homology long exact sequence of the pair $(X_{m^r},Y_{m^r})$,  $$H_i(X_{m^r};\Z_m)\cong H_i(Y_{m^r};\Z_m)$$ for every $i$. Since $Y_{m^r}$ is homeomorphic to $\bigvee_{i=1}^{m^r(n-1)+1}S^1$, $H_i(Y_{m^r})$ is trivial if $i\geq 2$. By the universal coefficient theorem, $H_i(X_{m^r})$ is a finite abelian group whose order is coprime to $m$ if $i\geq 2$. Since $X_{m^r}$ is connected, 
\[1-b_1(X_{m^r})=\chi(X_{m^r})=m^r\chi(X)=m^r\chi(Y)=m^r(1-n).\]
It follows that $b_1(X_{m^r})=m^r(n-1)+1$ and $b_i(X_{m^r})=0$ for $i\geq 2$. By the universal coefficient theorem, $X_{m^r}$ is a rational homology bouquet of $m^r(n-1)+1$ circles.
\end{proof}

Using Proposition~\ref{prop:Levine}, we prove the link analogue of the well-known  fact that every prime power fold branched cover of $S^3$ along a slice knot bounds a rational homology ball \cite[Lemma~2]{Casson-Gordon:1986-1}. 

\begin{proposition}\label{prop:betti-number-two-piece}
Suppose that $L$ is an oriented $n$-component slice link and $m$ is a prime. Let $\Sigma_{m^r}(L)$ be the $m^r$-fold cyclic branched cover of $S^3$ along $L$. Then  the following holds. 
\begin{enumerate}
\item \label{item:branchedcoverofdisk}$\Sigma_{m^r}(L)$ bounds a rational homology $\natural_{i=1}^{(m^r-1)(n-1)}S^1\times D^3$.
\item\label{item:branchedcoveroflink}  $\Sigma_{m^r}(L)$ is a rational homology $\#_{i=1}^{(m^r-1)(n-1)}S^1\times S^2$.
\end{enumerate}
\end{proposition} 
\begin{proof} \eqref{item:branchedcoverofdisk} Let $D$ be the union of slice disks for~$L$ and $\Sigma_{m^r}(D)$ be the $m^r$-fold cyclic cover of $D^4$ branched along $D$. Since $\partial \Sigma_{m^r}(D)=\Sigma_{m^r}(L)$, it suffices to show that $\Sigma_{m^r}(D)$ is a rational homology $\natural_{i=1}^{(m^r-1)(n-1)}S^1\times D^3$.

Let $X$ be the complement of $D$ in~$D^4$. By Alexander duality, a meridian map $f\colon \bigvee_{i=1}^n S^1\to X$ is a homology equivalence, and hence $X$ is a homology bouquet of $n$ circles. By Proposition~\ref{prop:Levine}, $X_{m^r}$ is a rational homology bouquet of $m^r(n-1)+1$ circles. In particular, 
\[b_k(X_{m^r})=\begin{cases}
 m^r(n-1)+1&\mbox{if $k=1$}\\
0&\mbox{if $k>1$.}\end{cases}\] By definition, $\Sigma_{m^r}(D)$ is obtained from $X_{m^r}$ by attaching $n$ 2-handles $H_i$. The attaching curve of each 2-handle $H_i$ is the lift of $m^r$ times the corresponding  meridian of~$D$ in~$X$.  By the Mayer-Vietoris sequence applied to $\Sigma_{m^r}(D)=\bigsqcup_{i=1}^nH_i\cup X_{m^r}$, the following is exact:
\[\bigoplus_{i=1}^n H_1(\d_+H_i )\to  H_1(X_{m^r})\to H_1(\Sigma_{m^r}(D))\to 0\]
where $\d_+H_i$ is the attaching region of the 2-handle $H_i$. The covering map $ X_{m^r}\to X$ induces a homomorphism $H_1(X_{m^r})\to H_1(X)$ whose image is the index $m^r$ subgroup, generated by $m^r$ times meridian of $L$. By restricting $H_1(X_{m^r})\to H_1(X)$ onto its image, we obtain a surjective map $H_1(X_{m^r})\to \Z^n$, which is a splitting of the composition $\Z^n\cong\bigoplus_{i=1}^n H_1(\d_+H_i )\to H_1(X_{m^r})$.  It follows that $H_1(X_{m^r})\cong \Z^n\oplus H_1(\Sigma_{m^r}(D))$, and hence
\[b_1(\Sigma_{m^r}(D))=b_1(X_{m^r})-n=m^r(n-1)+1-n=(m^r-1)(n-1).\]
Since $\d_+H_i\cong S^1\times D^2$, $H_k(\d_+H_i)=0$ for all $k>1$. Hence, by looking at the former part of the Mayer-Vietoris sequence, $H_k(\Sigma_{m^r}(D))\cong H_k(X_{m^r})$, and hence $b_k(\Sigma_{m^r}(D))= b_k(X_{m^r})=0$ for any $k>1$.  

 \eqref{item:branchedcoveroflink} By Poincar\'{e} duality and the universal coefficient theorem, 
 \[H_k(\Sigma_{m^r}(D),\Sigma_{m^r}(L);\Q)\cong H^{4-k}(\Sigma_{m^r}(D);\Q)\cong \operatorname{Hom}(H_{4-k}(\Sigma_{m^r}(D);\Q),\Q).\]
Hence, by the item \eqref{item:branchedcoverofdisk}, 
\[H_k(\Sigma_{m^r}(D),\Sigma_{m^r}(L);\Q)\cong \begin{cases}\Q &\mbox{if $k=4$}\\
\Q^{(m^r-1)(n-1)}&\mbox{if $k=3$}\\
0&\mbox{otherwise.}\end{cases}\]
From the homology long exact sequence of the pair $(\Sigma_{m^r}(D),\Sigma_{m^r}(L))$, 
\[H_k(\Sigma_{m^r}(L);\Q)\cong \begin{cases}\Q &\mbox{if $k=0,3$}\\
\Q^{(m^r-1)(n-1)}&\mbox{if $k=1,2$.}\end{cases}\]
Hence, $\Sigma_{m^r}(L)$ is a rational homology $\#_{i=1}^{(m^r-1)(n-1)}S^1\times S^2$.
\end{proof}

We now apply the above to obtain an obstruction in terms of morphisms of integral lattices, which we first define.

\begin{definition}[Integral lattices and morphisms]
An \emph{integral lattice} is a pair $(G,Q)$ where $G$ is a finitely generated free abelian group and $Q\colon G\times G\rightarrow \Z$ is a symmetric bilinear form.  Let $(G,Q)$ and $(G',Q')$ be integral lattices. A group homomorphism $\phi\colon G\rightarrow G'$ such that $Q(x,y)=Q'(\phi(x),\phi(y))$ for all $x,y\in G$ is a \emph{morphism} from $(G,Q)$ to $(G',Q')$ also denoted by $\phi$.
\end{definition}
Combined with \cite[Proposition~3.3]{Aceto:2020-1}, we give a link slicing obstruction. 
\begin{theorem}\label{thm:obstruction} 
Let $L$ be a slice link and $m^r$ be a prime power. Suppose that the $m^r$-fold cyclic branched cover of $S^3$ along $L$, denoted by 
$\Sigma_{m^r}(L)$, bounds a smooth, simply connected, positive semidefinite $4$-manifold $X$ such that $\operatorname{rank}(Q_X)=b_2(X)-b_1(\Sigma_{m^r}(L))$, then there is a morphism of integral lattices $$(H_2(X),Q_X)\rightarrow (\Z^{\operatorname{rank}(Q_X)}, \Id).$$
\end{theorem}
\begin{proof} By Proposition~\ref{prop:betti-number-two-piece}, $\Sigma_{m^r}(L)$  bounds a compact smooth $4$-manifold $Y$ which is a rational homology $\natural_{i=1}^{b_1(\Sigma_{m^r}(L))}S^1\times D^3$. There exists a morphism of integral lattices $(H_2(X),Q_X)\rightarrow (\Z^{\operatorname{rank}(Q_X)}, \Id)$ by \cite[Proposition~3.3]{Aceto:2020-1}. 
\end{proof}

Now we prove Theorem~\ref{thm:ribbon-up-to-mutation}. 

\begin{proof}[Proof of Theorem~\ref{thm:ribbon-up-to-mutation}]
We refer to \cite{Aceto:2020-1} for some notation and terminology used in this proof.
 Let $L=P(a,b,c,d)$ be a 2-component pretzel link, and assume that $L$ is slice. By Proposition \ref{prop:betti-number-two-piece},  $\Sigma_2(L)$ bounds a rational homology $S^1\times D^3$. 

Since $L$ is a $2$-component link, either none or exactly two of the parameters $a,b,c,d$ may be even. Suppose that one parameter is zero. Then it is easy to see that the other even parameter is twice the linking number between the two components of $L$ and thus must be zero since $L$ is slice. Then $L$ is either the split union of the unknot and $T_{2,n} \# T_{2,m}$ or the split union of $T_{2,n}$ and $T_{2,m}$ for some odd integers $n$, $m$. In particular, since $L$ is slice, $L$ is isotopic to $T_{2,n} \# T_{2,m}$. This link is slice if and only if $n= -m$ which can be easily seen by looking at the signature. Thus, in this case, $L$ is of the form $P(0,0,q,-q)$ for some odd integer $q$. 

From now on we may assume that the parameters $a,b,c,d$ are all non-zero. Suppose that exactly one parameter, say $a$, has absolute value one. In this case the branched double cover $\Sigma_2(L)$ is a Seifert fibered space over $S^2$ with three exceptional fibers. Then it follows from \cite[Theorem~5.5]{Aceto:2020-1} that $\Sigma_2(L)$ does not bound a rational homology~$S^1\times D^3$.

Now, suppose there is more than one coefficient with absolute value one. Then we can see through an isotopy that $L$ is a $2$-bridge link (alternatively, it is straightforward to see that $\Sigma_2(L)$ is a lens space, which implies that $L$ is a $2$-bridge link). Since the only lens space which is a rational homology $S^1\times S^2$ is $S^1\times S^2$ itself and lens spaces are double branched covers for a unique $2$-bridge link~\cite{Hodgson-Rubinstein:1985-1}, which is this case must be the $2$-component unlink. Thus, $L$ must be isotopic to $P(p,-p,1,-1)$ for some integer $p$. (Note that up to isotopy $\pm 1$ boxes can be put anywhere.)
 
From now on we may assume that all the coefficients have absolute value greater than one. The branched double cover
 $\Sigma_2(L)$ is a Seifert fibered space over $S^2$ with four exceptional fibers. More precisely we have
 $$
 \Sigma_2(L)=M(0;k,(a_1,a_2),(b_1,b_2),(c_1,c_2),(d_1,d_2))
 $$ 
 where $(a_1,a_2)=(a,1)$ if $a>0$ and $(a_1,a_2)=(a,a-1)$ if $a<0$ and similarly 
 for the other pairs. (Our notation coincides with that of \cite{Neumann-Raymond:seifert-manifolds}, see Figure~\ref{fig:Seifert} for a surgery diagram of $M(0;k,(a_1,a_2),(b_1,b_2),(c_1,c_2),(d_1,d_2))$.) \begin{figure}[htb]
\includegraphics[scale=1.2]{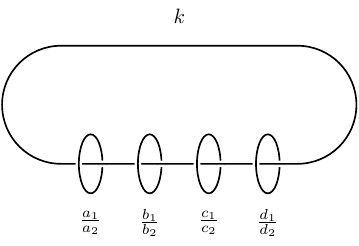}
\caption{A surgery diagram of $M(0;k,(a_1,a_2),(b_1,b_2),(c_1,c_2),(d_1,d_2))$.}\label{fig:Seifert}
\end{figure}The integer $k$ is the number of negative coefficients in 
 $\{a,b,c,d\}$. According to Theorem~5.5 in \cite{Aceto:2020-1}, $\Sigma_2(L)$ bounds a rational homology~$S^1\times D^3$ if and only if the Seifert invariants
 occur in complementary pairs and $e(\Sigma_2(L))=0$. Recall that the Seifert invariants $(a,b),(c,d)$ are said to be complementary 
 if $(c,d)=(a,a-b)$. In our situation these two 
 conditions are equivalent to $\Sigma_2(L)$ being of the form
 $$
 \Sigma_2(L)=M(0;2,(p,1),(p,p-1),(q,1),(q,q-1))
 $$
 for some $p,q>1$. From this description we can recover the coefficients defining 
 $L$. In particular we see that, possibly after renaming the coefficients,
 we have $\{a,b,c,d\}=\{p,-p,q,-q\}$. Recall that for any pretzel link $P(a_1,\dots,a_n)$ its isotopy class only depends on the ordered string of integers $(a_1,\dots,a_n)$ up to cyclic permutation and overall reversal of order. Thus $L$ is isotopic to either $P(p,-p,q,-q)$, $P(p,-p,-q,q)$, or $P(p,q,-p,-q)$. This completes the proof of Theorem~\ref{thm:ribbon-up-to-mutation}.
\end{proof}

%

%



\section{Proof of Theorem~\ref{thm:mutants-not-concordant}}\label{sec:proof}
Let $L$ denote the 2-component pretzel link $P(p,q,-p,-q)$, where $p$ and $q$ are distinct integers greater than one. Note that this link has two components if and only if both $p$ and $q$ are odd or, without loss of generality, $p$ is odd and $q$ is even. If $p$ is odd and $q$ is even each component of $L$ is a non-trivial torus knot, and thus $L$ is not slice. Hence, we may assume that $p$ and $q$ are odd integers such that $1<p<q$. Our goal is to prove that $L$ is not slice. In Section~\ref{subsection:constructionofX}, we show that $b_1(\Sigma_3(L))=2$ and that $\Sigma_3(L)$ bounds a compact, simply connected, positive definite 4-manifold $X$ such that $b_2(X)=\frac{3(p+q)}{2}+2$ and $\operatorname{rank}(Q_X)=b_2(X)-b_1(\Sigma_{3}(L))$. In Section~\ref{subsection:lattice}, we show that there is no lattice morphism $(H_2(X),Q_X)\rightarrow (\Z^{\operatorname{rank}(Q_X)}, \Id)$. By Theorem~\ref{thm:obstruction}, this will complete the proof of Theorem~\ref{thm:mutants-not-concordant}. 
\subsection{Construction of the 4-manifold $X$}\label{subsection:constructionofX}
\begin{figure}[htb!]
\includegraphics[scale=0.7]{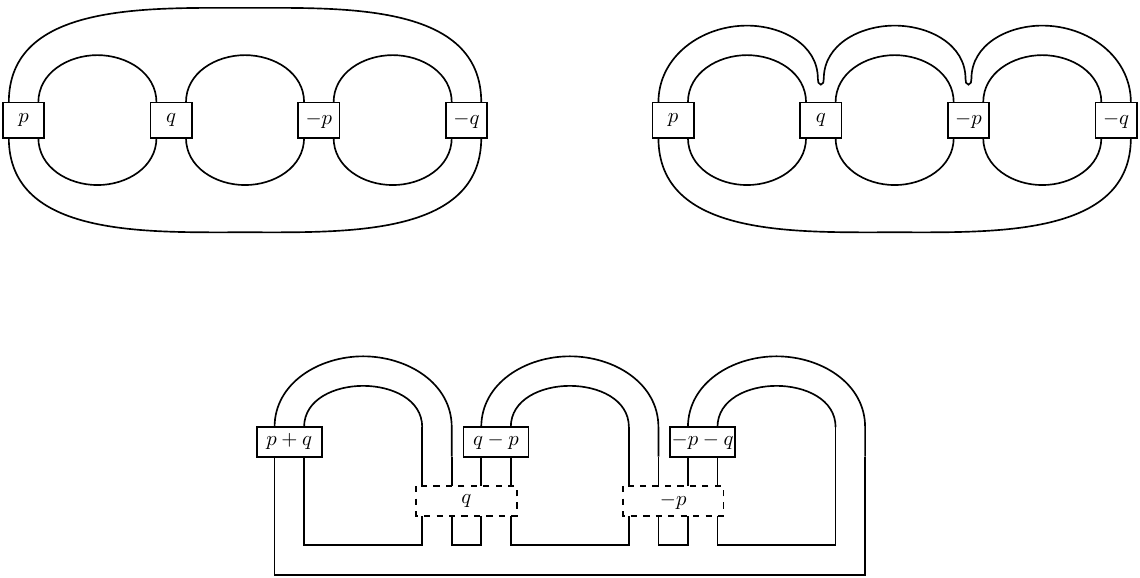}
\caption{A sequence of isotopies. Top left:\ the pretzel link $P(p,q,-p,-q)$. Bottom:\ An orientable Seifert surface has become visible. Throughout, the solid boxes indicate half-twists between all the strands passing through. In contrast, the dashed boxes indicate half-twists solely between the bands.}\label{fig:seifert_surface}
\end{figure}

Our first task is to build an orientable Seifert surface for $L$, which will assist us in drawing a surgery diagram for $\Sigma_3(L)$, the $3$-fold branched cover of $S^3$ along $L$. 
The link~$L$ is given in the top panel of Figure~\ref{fig:seifert_surface}. Perform the isotopy indicated in the second panel, which consists of pushing strands of $L$ through the two middle boxes. This results in the third picture, where the dashed boxes indicate half-twists between the bands and the solid boxes indicate half-twists between the two strands passing through. In the last diagram, an orientable Seifert surface for $L$ has become visible. Now, the standard process from~\cite{Akbulut-Kirby:branched-covers} gives a surgery diagram for $\Sigma_3(L)$, shown in Figure~\ref{fig:surgery_diagram}. We label the six curves in the surgery diagram $u_i$, $v_i$, $w_i$ ($i=1,2$) as indicated in the figure. The linking-framing matrix for this diagram is shown below. By row-reducing the matrix, we see that  $b_1(\Sigma_3(L))=2$.

\begin{figure}[htb]
\includegraphics[scale=1]{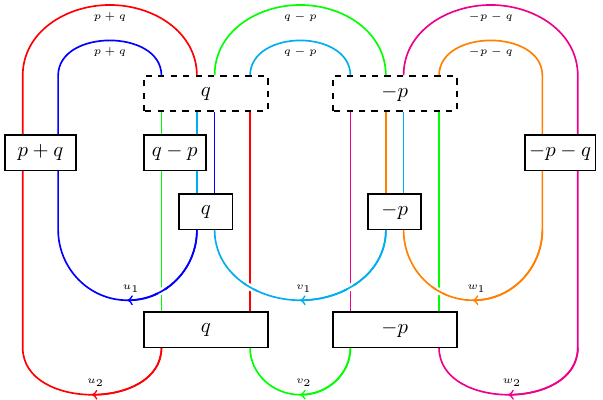}
\caption{A surgery diagram for $\Sigma_3(L)$. Note that the numbers in the boxes denote half-twists.}\label{fig:surgery_diagram}
\end{figure}

 \[\begingroup 
\setlength{\arraycolsep}{7pt}
\setlength{\kbrowsep}{7pt}
\kbordermatrix{
&\phantom{hh}u_1\phantom{hh}&\phantom{hh}u_2\phantom{hh}&\phantom{hh}v_1\phantom{hh}&\phantom{hh} v_2\phantom{hh}&\phantom{hh}w_1\phantom{hh}&\phantom{hh}w_2\phantom{hh}\cr
u_1&p+q&\frac{p+q}{2}&-q&\frac{-q+1}{2}&0&0\cr \\
u_2&\frac{p+q}{2}&p+q&\frac{-q-1}{2}&-q&0&0\cr \\
v_1&-q&\frac{-q-1}{2}&q-p&\frac{q-p}{2}&p&\frac{p+1}{2}\cr \\
v_2&\frac{-q+1}{2}&-q&\frac{q-p}{2}&q-p&\frac{p-1}{2}&p\cr \\
w_1&0&0&p&\frac{p-1}{2}&-p-q&\frac{-p-q}{2}\cr \\
w_2&0&0&\frac{p+1}{2}&p&\frac{-p-q}{2}&-p-q\cr
}
\endgroup
 \]

\begin{figure}[htb]
\includegraphics[scale=1]{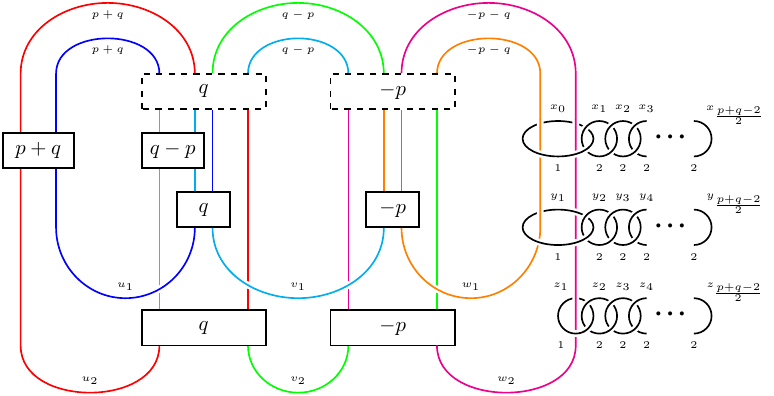}
\caption{A Kirby diagram for $X\# 2\overline{\mathbb{C}\mathbb{P}^2}$ whose boundary is $\Sigma_3(L)$. }\label{fig:building_X}
\end{figure}

\vspace{0.8cm}
Next, we construct a simply connected $4$-manifold $X$ with $\partial X=\Sigma_3(L)$ such that $(H_2(X),Q_X)$ is positive semidefinite with $\rk(Q_X)=b_2(X)-2$. The rough idea of the construction is to first look at the 3-manifold described by the sublink given by $(w_1, w_2)$ in Figure~\ref{fig:surgery_diagram}. This 3-manifold is a Seifert fibered space over the 2-sphere with 3 exceptional fibers. Such a space has a unique plumbing description given by its positive canonical plumbing graph. In our case, this plumbing graph is positive definite. We change the surgery description of $\Sigma_3(L)$ given by Figure~\ref{fig:surgery_diagram} so that the 2-component link $(w_1,w_2)$ is replaced by the surgery corresponding to the associated plumbing graph. The resulting surgery description of $\Sigma_3(L)$ also describes the simply connected $4$-manifold $X$ with $\partial X=\Sigma_3(L)$ since each surgery coefficient is an integer. We now give the details of this construction.

Blow up at a chain of curves, labeled $x_0,x_1,\dots, x_{\frac{p+q-2}{2}}$ to unlink $w_1$ and $w_2$. Next, blow up at two chains of curves, labeled $y_1,\dots, y_{\frac{p+q-2}{2}}$ and $z_1,\dots, z_{\frac{p+q-2}{2}}$ in the figure, to change the framing on $w_1$ and $w_2$ to~$-1$. The linking-framing matrix for this new diagram, shown in Figure~\ref{fig:building_X}, is given below. Only non-zero entries are shown.
\renewcommand{\arraystretch}{1.5}

\begin{Tiny}
 \[\begingroup 
\setlength{\arraycolsep}{1.3pt}
\setlength{\kbrowsep}{10pt}
\kbordermatrix{
&	u_1		&	u_2		&	& x_0	&	x_1	&	\cdots	&	x_{\frac{p+q-2}{2}}	&& y_1&y_2	&\cdots	& y_{\frac{p+q-2}{2}}&	&z_1	&z_2&\cdots&z_{\frac{p+q-2}{2}}&	&v_1&v_2&&w_1&w_2\cr
&p+q	&	\frac{p+q}{2}	&\vrule&	&		&		&		&	\vrule&	&		&		&		&\vrule&		&		&		&		&\vrule&	-q	&	\frac{-q+1}{2}	&\vrule&		&		\cr
&	\frac{p+q}{2}	&	p+q	&	\vrule&&		&		&		&	\vrule&	&		&		&		&	\vrule&	&		&		&		&\vrule&	\frac{-q-1}{2}	&	-q	&\vrule&		&		\cr
\BAhhline{&-----------------------}
&		&		&	\vrule&1	& 	1	&		&		&	\vrule&	&		&		&		&	\vrule&	&		&		&		&\vrule&		&		&\vrule&	1	&	1\cr
&		&		&\vrule&	1	&	2	&	\ddots	&		&	\vrule&	&		&		&		&	\vrule&	&		&		&		&\vrule&		&		&\vrule&		&	\cr
&		&		&	\vrule&	& 	\ddots	&	\ddots	&	1	&	\vrule&	&		&		&		&	\vrule&	&		&		&		&\vrule&		&		&\vrule&		&	\cr
&		&		&	\vrule&	& 		&	1	&	2	&	\vrule&	&		&		&		&	\vrule&	&		&		&		&\vrule&		&		&\vrule&		&	\cr
\BAhhline{&-----------------------}
&		&		&	\vrule&	&	&		&		&	\vrule&1	&	1	&		&		&	\vrule&	&		&		&		&\vrule&		&		&\vrule&	1	&	\cr
&		&		&	\vrule&	&		&		&		&\vrule&	1	&	2	&	\ddots	&		&	\vrule&	&		&		&		&\vrule&		&		&\vrule&		&	\cr
&		&		&	\vrule&	& 		&		&		&	\vrule&	&	\ddots	&	\ddots	&	1	&	\vrule&	&		&		&		&\vrule&		&		&\vrule&		&	\cr
&		&		&	\vrule&	& 		&		&		&	\vrule&	&		&	1	&	2	&	\vrule&	&		&		&		&\vrule&		&		&\vrule&		&	\cr
\BAhhline{&-----------------------}
&		&		&	\vrule&	& 		&		&		&	\vrule&	&		&		&		&\vrule&	1	&	1	&		&		&\vrule&		&		&\vrule&		&	1\cr
&		&		&	\vrule&	& 		&		&		&	\vrule&	&		&		&		&\vrule&	1	&	2	&	\ddots	&		&\vrule&		&		&	\vrule&	&	\cr
&		&		&	\vrule&	& 		&		&		&	\vrule&	&		&		&		&	\vrule&	&	\ddots	&	\ddots	&	1	&\vrule&		&		&\vrule&		&	\cr
&		&		&	\vrule&	& 		&		&		&	\vrule&	&		&		&		&	\vrule&	&		&	1	&	2	&\vrule&		&		&	\vrule&	&	\cr
\BAhhline{&-----------------------}
&	-q	&	\frac{-q-1}{2}	& 	\vrule&	&		&		&		&	\vrule&	&		&		&		&	\vrule&	&		&		&		&\vrule&	q-p	&	\frac{q-p}{2}	&\vrule&	p	&	\frac{p+1}{2}\cr
&	\frac{-q+1}{2}	&	-q	& 	\vrule&	&		&		&		&	\vrule&	&		&		&		&	\vrule&	&		&		&		&\vrule&	\frac{q-p}{2}	&	q-p	&\vrule&	\frac{p-1}{2}	&	p\cr
\BAhhline{&-----------------------}
&		&		&	\vrule&1	& 		&		&		&	\vrule&1	&		&		&		&	\vrule&	&		&		&		&\vrule&	p	&	\frac{p-1}{2}	&	\vrule&-1	&	\cr
&		&		&	\vrule&1	& 		&		&		&	\vrule&	&		&		&		&	\vrule&1	&		&		&		&\vrule&	\frac{p+1}{2}	&	p	&\vrule&		&	-1\cr
}
\endgroup
 \]
\end{Tiny}

Returning to Figure~\ref{fig:building_X} we see that that $w_1$ and $w_2$ are unknotted curves with framing~$-1$. Since they are unlinked, we can blow down both of them. The resulting manifold is what we call $X$. Hence, the Kirby diagram in Figure~\ref{fig:building_X} represents $X\#2\overline{\mathbb{C}\mathbb{P}^2}$. Since only blow ups and blow downs were performed, we see that $\partial X=\Sigma_3(L)$ as needed. Note that the linking-framing matrix for the Kirby diagram for $X$ obtained from the second diagram in Figure~\ref{fig:building_X} by blowing down $w_1$ and $w_2$  has dimension $\frac{3(p+q)}{2}+2$. Let $n$ denote the quantity $\frac{3(p+q)}{2}+2$ and let $Q_X$ denote the linking-framing matrix of this diagram. Rather than explicitly drawing the Kirby diagram, we can simply use the information from the linking-framing matrix on the previous page, to find $Q_X$, which is given below. 

\begin{Tiny}
 \[\begingroup 
\setlength\arraycolsep{1.3pt}
\setlength{\kbrowsep}{10pt}
\kbordermatrix{&	u_1	&	u_2	&&	x_0	&	x_1	&	\cdots	&	x_{\frac{p+q-2}{2}}	&&	y_1	&	y_2		&	\cdots	&	y_{\frac{p+q-2}{2}}	&&	z_1	&	z_2	&	\cdots	&	z_{\frac{p+q-2}{2}}	&&	v_1	&	v_2	\\
&	p+q	&	\frac{p+q}{2}	&	\vrule&	&		&		&		&\vrule&		&		&		&		&\vrule&	&		&		&		&\vrule&	-q	&	\frac{-q+1}{2}	\\
&	\frac{p+q}{2}	&	p+q	&	\vrule&	&		&		&		&\vrule&		&		&		&		&\vrule&		&		&		&		&\vrule&	\frac{-q-1}{2}	&	-q	\\
\BAhhline{&--------------------}
&		&		&	\vrule&3	&	1	&		&		&	\vrule&1	&		&		&		&\vrule&	1	&		&		&		&\vrule&	\frac{3p+1}{2}	&	\frac{3p-1}{2}	\\
&		&		&	\vrule&1	&	2	&	\ddots	&		&	\vrule&	&		&		&		&\vrule&		&		&		&		&\vrule&		&		\\
&		&		&	\vrule&	&	\ddots	&	\ddots	&	1	&	\vrule&	&		&		&		&\vrule&		&		&		&		&\vrule&		&		\\
&		&		&	\vrule&	&		&	1	&	2	&	\vrule&	&		&		&		&\vrule&		&		&		&		&\vrule&		&			\\
\BAhhline{&--------------------}
&		&		&	\vrule&1	&		&		&		&\vrule&	2	&	1	&		&		&\vrule&		&		&		&		&\vrule&	p	&	\frac{p-1}{2}\\
&		&		&	\vrule&	&		&		&		&	\vrule&1	&	2	&	\ddots	&		&\vrule&		&		&		&		&\vrule&		&		\\
&		&		&	\vrule&	&		&		&		&	\vrule&	&	\ddots	&	\ddots	&	1	&\vrule&		&		&		&		&\vrule&		&			\\
&		&		&	\vrule&	&		&		&		&	\vrule&	&		&	1	&	2	&\vrule&		&		&		&		&\vrule&		&		\\
\BAhhline{&--------------------}
&		&		&	\vrule&1	&		&		&		&	\vrule&	&		&		&		&\vrule&	2	&	1	&		&		&\vrule&	\frac{p+1}{2}	&	p	\\
&		&		&	\vrule&	&		&		&		&	\vrule&	&		&		&		&\vrule&	1	&	2	&	\ddots	&		&\vrule&		&		\\
&		&		&	\vrule&	&		&		&		&	\vrule&	&		&		&		&\vrule&		&	\ddots	&	\ddots	&	1	&\vrule&		&		\\
&		&		&	\vrule&	&		&		&		&	\vrule&	&		&		&		&\vrule&		&		&	1	&	2	&	\vrule&	&	\\
\BAhhline{&--------------------}
&	-q	&	\frac{-q-1}{2}	&\vrule&	\frac{3p+1}{2}	&		&		&		&\vrule&	p	&		&		&		&\vrule&	\frac{p+1}{2}	&		&		&		&\vrule&	\frac{5p^2-2p+4q+1}{4}	&	\frac{2p^2-p+q}{2}	\\
&	\frac{-q+1}{2}	&	-q	&	\vrule&\frac{3p-1}{2}	&		&		&		&\vrule&	\frac{p-1}{2}	&		&		&		&\vrule&	p	&		&		&		&\vrule&	\frac{2p^2-p+q}{2} &	\frac{5p^2-6p+4q+1}{4}	\\
}
\endgroup
 \]
\end{Tiny}

\begin{figure}[htb]
\includegraphics[scale=0.8]{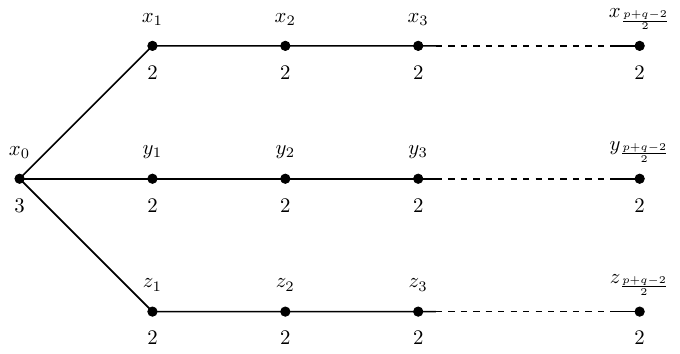}
\caption{The plumbing graph corresponding to  $P$.}\label{fig:plumbed-manifold}
\end{figure} 
 Since $Q_X$ is a presentation matrix for $H_1(\Sigma_3(L);\Q)$, we see that  $\rk (Q_X)=n-2$. Next, we will show that $b_2^+(X)\geq n-2$. Consider the $(n-4)\times (n-4)$ submatrix of $Q_X$ corresponding to the curves 
\[x_0, x_1,\dots, x_{\frac{p+q-2}{2}}, y_1,\dots, y_{\frac{p+q-2}{2}}, z_1,\dots, z_{\frac{p+q-2}{2}}.\] This is the intersection matrix for the plumbed manifold $P$ given in Figure~\ref{fig:plumbed-manifold}. By~\cite[Theorem~5.1]{Neumann-Raymond:seifert-manifolds}, $\partial P$ is the Seifert fibered manifold given by 
\[
M(0; 3, (p+q,p+q-2), (p+q,p+q-2), (p+q,p+q-2)),
\]
which has Euler number
\[
3-\frac{3(p+q-2)}{p+q}=\frac{6}{p+q}>0.
\]
By~\cite[Theorem~5.2]{Neumann-Raymond:seifert-manifolds}, this implies that the intersection form of $P$ is positive definite. 
Additionally, the top left $2\times 2$ submatrix of $Q_X$ is positive definite since it has both positive trace and positive determinant. Moreover, note that the curves $u_1$ and $u_2$ do not link $x_0, x_1,\dots, x_{\frac{p+q-2}{2}}, y_1,\dots, y_{\frac{p+q-2}{2}}, z_1,\dots, z_{\frac{p+q-2}{2}}$. This implies that the top left $(n-2)\times (n-2)$ submatrix of $Q_X$ is positive definite, and thus $b_2^+(X)\geq n-2$. Since $\rk(Q_X)=n-2$ and $b_2^+(X)\geq n-2$, $(H_2(X),Q_X)$ is positive semidefinite with $\rk(Q_X)=b_2(X)-2$.

\subsection{Non-existence of a lattice morphism $(H_2(X),Q_X)\rightarrow (\Z^{\operatorname{rank}(Q_X)}, \Id)$}\label{subsection:lattice}

For the sake of contradiction, suppose that there is a lattice morphism 
\[(H_2(X),Q_X)\rightarrow (\Z^{\operatorname{rank}(Q_X)}, \Id).\]
 Recall that $H_2(X)$ is generated by the $2$-handles associated to the curves 
\[
u_1, u_2, x_0, x_1,\dots, x_{\frac{p+q-2}{2}}, y_1,\dots, y_{\frac{p+q-2}{2}}, z_1,\dots, z_{\frac{p+q-2}{2}},v_1,v_2
\]
and $\operatorname{rank}(Q_X)=\frac{3(p+q)}{2}$. To avoid a profusion of symbols, we will use the same notation for the corresponding homology classes as well as their images via the lattice morphism $(H_2(X),Q_X)\rightarrow (\Z^{\operatorname{rank}(Q_X)}, \Id)$. Let 
\[\{e_1,\dots, e_\frac{p+q}{2}, f_1,\dots, f_\frac{p+q}{2}, g_1,\dots, g_\frac{p+q}{2}\}\]
 form a standard basis for $(\Z^{\operatorname{rank}(Q_X)},\Id)$, that is, 
 \begin{align*}e_i\cdot e_j&=f_i\cdot f_j=g_i\cdot g_j=\delta_{ij}\\
e_i\cdot f_j&=f_i\cdot g_j=g_i\cdot e_j=0 \end{align*}
for all $i,j$.  First consider the class $x_1\in H_2(X)$. Since the norm of $x_1$ is $2$, we can write 
\[x_1=e_1+e_2,\]
 up to relabeling the elements of the standard basis. Similarly, we can write 
 \begin{align*}
 x_2&=e_2+e_3,\\
  y_1&=f_1+f_2,\\
  y_2&=f_2+f_3,\\
   z_1&=g_1+g_2,\\
   z_2&=g_2+g_3,
 \end{align*}where we are using the fact that $ x_i\cdot y_j=x_i\cdot z_k= y_j\cdot z_k=0$ if $i,j,k>0$ and $x_1\cdot x_2=y_1\cdot y_2=z_1\cdot z_2=1$. Since $x_0$ has norm $3$, $x_0\cdot x_1=x_0\cdot y_1=x_0\cdot z_1=1$, and $x_0\cdot x_2=x_0\cdot y_2=x_0\cdot z_2=0$, the only possibility is $x_0=e_1+f_1+g_1$.
From here, it is easy to see that we can write $x_i= e_i+e_{i+1}$, $y_i=f_i+f_{i+1}$, and $
z_i= g_i+g_{i+1}$ for all $i=1,\dots, \frac{p+q-2}{2}$. 

We introduce the following notation.
\begin{align*}
e&=e_1-e_2+e_3-\cdots + (-1)^{\frac{p+q}{2}}e_{\frac{p+q}{2}},\\
f&=f_1-f_2+f_3-\cdots + (-1)^{\frac{p+q}{2}}f_{\frac{p+q}{2}},\\
g&=g_1-g_2+g_3-\cdots + (-1)^{\frac{p+q}{2}}g_{\frac{p+q}{2}}.\\
\end{align*}
Clearly, $e\cdot e=f\cdot f=g \cdot g=\frac{p+q}{2}$ and $e\cdot f=e\cdot g=f \cdot g=0$.

Note that we have already used all the available coordinates. We still need to determine the image of the sublattice generated by $u_1$ and $u_2$. Using the fact that
this sublattice is orthogonal to the one associated to the plumbing graph given in Figure~\ref{fig:plumbed-manifold} one quickly obtains
$u_1=\varepsilon_1 e+ \varepsilon_2 f +\varepsilon_3 g$, where $\varepsilon_i\in \{-1,0,1\}$ and exactly two of the $\varepsilon_i$ are non-zero with different signs. 
Note that $u_2$ must be of the same form. Now, considering that $u_1\cdot u_2 = \frac{p+q}{2}$, we see that the following are the possible values:
\begin{equation}\label{tab:L1-L2-values}
  \begin{array}{ r | c | c | c | c | c | c }
u_1        &e-f        &e-g        &f-e        &f-g        &g-e        &g-f\\
\hline
\multirow{2}{*}{$u_2$}       &e-g        &e-f        &f-g        &f-e        &g-f        &g-e\\
&g-f        &f-g        &g-e        &e-g        &f-e        &e-f\\
  \end{array}
\end{equation}
Above, one should read the table as saying that if $u_1=e-f$ then $u_2=e-g$ or $g-f$, if $u_1=e-g$ then $u_2=e-f$ or $f-g$, etc.

It remains to consider the curves $v_1$ and $v_2$. Since $v_1$ does not link any of the $x_i$, $y_j$, and $z_j$ curves for $i=1,\ldots,\frac{p+q}{2}$, $j=2,\ldots,\frac{p+q}{2}$, it is straightforward to see that $v_1= ae+bf+cg+d_1f_1+d_2g_1$ for some $a,b,c,d_1,d_2$. Moreover, from the matrix $Q_X$, we know that 
\begin{align*}&v_1\cdot (e_1+f_1+g_1)=\frac{3p+1}{2},\\
&v_1\cdot (f_1+f_2)=p,\\
&v_1\cdot (g_1+g_2)=\frac{p+1}{2}\end{align*}
which implies that $d_1=p$, $d_2=\frac{p+1}{2}$, and $a+b+c=0$. We have the following further restrictions on $v_1$:
\begin{align*}
&v_1\cdot u_1	=-q,\\
&v_1\cdot u_2	=\frac{-q-1}{2}.\\
\end{align*}
Similarly, we see that $v_2= a'e+b'f+c'g+d'_1f_1+d'_2g_1$ for some $a',b',c',d_1',d_2'$. We now have that 
\begin{align*}
&v_2\cdot (e_1+f_1+g_1)	=\frac{3p-1}{2},\\
&v_2\cdot (f_1+f_2)	=\frac{p-1}{2},\\
&v_2\cdot (g_1+g_2)	=p.
\end{align*}
which implies that $d_1'=\frac{p-1}{2}$, $d_2'= p$, and $a'+b'+c'=0$. We also have the following restrictions on $v_2$:
\begin{align*}
v_2\cdot u_1	&=\frac{-q+1}{2},\\
v_2\cdot u_2	&=-q.\\
\end{align*}
At this point, we have the following system of six linear equations 
\begin{align*}
&v_1\cdot u_1	=-q,\\
&v_1\cdot u_2	=\frac{-q-1}{2},\\
&v_2\cdot u_1	=\frac{-q+1}{2},\\
&v_2\cdot u_2	=-q,\\
&a+b+c	=0,\\
&a'+b'+c'=0.
\end{align*}
in the variables $a,b,c, a',b',c'$, where 
\begin{align*}&v_1= ae+bf+cg + pf_1+ \left(\frac{p+1}{2}\right)g_1,\\
 &v_2=a'e+b'f+c'g+\left(\frac{p-1}{2}\right)f_1+pg_1,\\
  &e\cdot e=f\cdot f=g \cdot g=\frac{p+q}{2},\\
  &e\cdot f=e\cdot g=f \cdot g= e\cdot f_1=e\cdot g_1=f\cdot g_1=g\cdot f_1=0,\\
  &f\cdot f_1 = g\cdot g_1 =1,
  \end{align*} and there are $12$ possible values of $\{u_1,u_2\}$ as given in Table~\eqref{tab:L1-L2-values}. Elementary linear algebra gives the following corresponding values of $a,b,c,a',b',c'$.
\begin{center}
\renewcommand{\arraystretch}{1.7}
$
\begin{array}{!{\vline width 1.1pt} c | c !{\vline width 1.1pt} c | c | c !{\vline width 1.1pt}c | c | c !{\vline width 1.1pt}}
\Xhline{1.1pt}
u_1	&u_2	&a	&b	&\hphantom{000}c\hphantom{000}	&a'	&\hphantom{000}b'\hphantom{000}	&c'\\
\Xhline{1.1pt}
e-f	&e-g	&\frac{p-q}{p+q}	&\frac{-p+q}{p+q}	&0	&\frac{p-q}{p+q}	&0	&\frac{-p+q}{p+q}\\
\hline
e-f	&g-f	&\frac{2+3(p-q)}{3(p+q)}	&\frac{2-3(p-q)}{3(p+q)}	&\frac{-4}{3(p+q)}	&\frac{1+3p}{3(p+q)}	&\frac{1+3q}{3(p+q)}	&\frac{-2-3(p+q)}{3(p+q)}\\
\hline
e-g	&e-f	&\frac{p-q}{p+q}	&\frac{1-p}{p+q}	&\frac{-1+q}{p+q}	&\frac{p-q}{p+q}	&\frac{1+q}{p+q}	&\frac{-1-p}{p+q}\\
\hline
e-g	&f-g	&\frac{2+3(p-q)}{3(p+q)}	&\frac{-1-3p}{3(p+q)}	&\frac{-1+3q}{3(p+q)}	&\frac{1+3p}{3(p+q)}	&\frac{1-3q}{3(p+q)}	&\frac{-2-3(p-q)}{3(p+q)}\\
\hline
f-e	&f-g	&1	&-1	&0	&\frac{-1+p}{p+q}	&\frac{1-q}{p+q}	&\frac{-p+q}{p+q}\\
\hline
f-e	&g-e	&\frac{2+3(p+q)}{3(p+q)}	&\frac{2-3(p+q)}{3(p+q)}	&\frac{-4}{3(p+q)}	&\frac{-2+3(p+q)}{3(p+q)}	&\frac{4}{3(p+q)}	&\frac{-2-3(p+q)}{3(p+q)}\\
\hline
f-g	&f-e	&\frac{1+p}{p+q}	&-1	&\frac{-1+q}{p+q}	&1	&\frac{1-q}{p+q}	&\frac{-1-p}{p+q}\\
\hline
f-g	&e-g	&\frac{-1+3p}{3(p+q)}	&\frac{2-3(p+q)}{3(p+q)}	&\frac{-1+3q}{3(p+q)}	&\frac{-2+3(p-q)}{3(p+q)}	&\frac{4}{3(p+q)}	&\frac{-2-3(p-q)}{3(p+q)}\\
\hline
g-e	&g-f	&1	&\frac{1-p}{p+q}	&\frac{-1-q}{p+q}	&\frac{-1+p}{p+q}	&\frac{1+q}{p+q}	&-1\\
\hline
g-e	&f-e	&\frac{2+3(p+q)}{3(p+q)}	&\frac{-1-3p}{3(p+q)}	&\frac{-1-3q}{3(p+q)}	&\frac{-2+3(p+q)}{3(p+q)}	&\frac{1-3q}{3(p+q)}	&\frac{1-3p}{3(p+q)}\\
\hline
g-f	&g-e	&\frac{1+p}{p+q}	&\frac{-p+q}{p+q}	&\frac{-1-q}{p+q}	&1	&0	&-1\\
\hline
g-f	&e-f	&\frac{-1+3p}{3(p+q)}	&\frac{2-3(p-q)}{3(p+q)}	&\frac{-1-3q}{3(p+q)}	&\frac{-2+3(p-q)}{3(p+q)}	&\frac{1+3q}{3(p+q)}	&\frac{1-3p}{3(p+q)}\\
\Xhline{1.2pt}
\end{array}
$
\end{center}
\vspace{1cm}

\noindent There is no integer solution set since $p$ and $q$ are distinct positive integers, and thus, we have reached the desired contradiction.
\bibliographystyle{amsalpha}
\def\MR#1{}
\bibliography{bib}

\end{document}